\documentclass[10pt]{amsart}

\usepackage{amssymb}
\usepackage{amsfonts,latexsym}

\usepackage[colorlinks=true]{hyperref}
\hypersetup{linkcolor=black,citecolor=black,filecolor=black,urlcolor=black} 

\newtheorem{theorem}{Theorem}[section]
\newtheorem{proposition}{Proposition}[section]

\newtheorem{lemma}{Lemma}[section]
\newtheorem{remark}{Remark}[section]
\newtheorem{definition}{Definition}[section]

\begin{document}

\title[$4$-ended solutions]{The space of $4$-ended solutions to the Allen-Cahn equation in the plane}

\author{Micha{\l } Kowalczyk}
\address{Micha{\l } Kowalczyk, Departamento de Ingenier\'{\i}a Matem\'atica and Centro
de Modelamiento Matem\'atico (UMI 2807 CNRS), Universidad de Chile, Casilla
170 Correo 3, Santiago, Chile.}
\email {kowalczy@dim.uchile.cl}

\author{Yong Liu}
\address{\noindent Y. Liu - Departamento de Ingenier\'{\i}a Matem\'atica and Centro de
Modelamiento Matem\'atico (UMI 2807 CNRS), Universidad de Chile, Casilla 170
Correo 3, Santiago, Chile.}
\email{yliu@dim.uchile.cl}

\author{Frank Pacard}
\address{Frank Pacard, Centre de Math\'ematiques Laurent Schwartz, \'Ecole
Polytechnique, 91128 Palaiseau, France et Institut Universitaire de France}
\email{frank.pacard@math.polytechnique.fr}

\maketitle

\begin{abstract}
We are interested in entire solutions of the Allen-Cahn equation $\Delta u- F^{\prime}(u) = 0$ which have some special structure at infinity. In this equation, the function $F$ is an even, bistable function. The solutions we are interested in have their zero set asymptotic to $4$ half oriented affine lines  at infinity and, along each of these half affine lines, the solutions are asymptotic to the one dimensional heteroclinic solution~: such solutions are called {\em $4$-ended solutions}. The main result of our paper states that, for any $\theta \in (0, \pi/2)$, there exists a $4$-ended solution of the Allen-Cahn equation whose zero set is at infinity asymptotic to the half oriented affine lines making the angles $\theta , \pi - \theta, \pi+\theta$ and $2\pi - \theta$ with the $x$-axis. This paper is part of a program whose aim is to classify all $2k$-ended solutions of the Allen-Cahn equation in dimension $2$, for $k \geq 2$. 
\end{abstract}

\section{Introduction}

In this paper, we are interested in entire solutions of the Allen-Cahn equation
\begin{equation}
\Delta u - F^{\prime}\left(  u\right)  = 0 ,
\label{AC}
\end{equation}
in $\mathbb R^2$, where the function $F$ is a smooth, double well potential. This means that $F$ is even, nonnegative and has only two zeros which will be chosen to be at $\pm 1$. Moreover, we assume that 
\[
F'' ( \pm 1 )  \neq 0 ,
\]
and also  that 
\[
F'(t)  \neq 0,\qquad \mbox{for all} \qquad t \in (0,1).
\]

\medskip

It is known that (\ref{AC}) has a solution whose nodal set is {\em any} given straight line. These special solutions, which will be referred to as the {\em heteroclinic solutions}, are constructed using the heteroclinic, one dimensional solution of (\ref{AC}), namely the function $H$ defined on $\mathbb R$, solution of 
\begin{align}
\label{heteroclinic}
H^{\prime\prime} - F^{\prime}(H) =0 ,
\end{align}
which is odd and tends to $-1$ (respectively to $+1$) at $-\infty$ (respectively at $+\infty$). 

\medskip

More precisely, we have the~:
\begin{definition}
Given $r\in \mathbb R$ and ${\tt e} \in \mathbb R^2$ such that $|{\tt e}|=1$, the {\em heteroclinic solutions with end} 
$\lambda : =  r \, {\tt e}^\perp + \mathbb R \, {\tt e}$ is defined by
\[
u({\tt x}) : = H({\tt x} \cdot {\tt e}^\perp - r),
\]
where $\perp$ denotes the rotation of angle $\pi/2$ in the plane. 
\end{definition}
Observe that this construction extends in any dimension to produce solutions whose level sets are hyperplanes. 

\medskip

The famous {\em de Giorgi conjecture} asserts that (in space dimension less than or equal to $8$), if $u$ is a bounded  solution of (\ref{AC}) which is monotone in one direction, then $u$ has to be one of the above defined heteroclinic solutions. This conjecture is known to hold when the space dimension is equal to $2$ \cite{MR1637919}, in dimension $3$ \cite{MR1775735}  and in dimension $4$ to $8$ \cite{MR2480601} under some mild additional assumption. Counterexamples have been constructed in all dimensions  $N\geq 9$ in \cite{dkp_dg}, showing that the conjecture is indeed sharp. 

\medskip

In this paper, we are interested in entire solutions of (\ref{AC}) which are defined in $\mathbb R^2$ and which have some special structure at infinity, namely their zero set is, at infinity, asymptotic to $4$ oriented affine lines~: such solutions are called {\em $4$-ended solutions} and will be precisely defined in the next section. The main result of our paper states that, for any $\theta \in (0, \pi/2)$, there exists a $4$-ended solution of (\ref{AC}) whose zero set is asymptotic  at infinity to the half oriented affine lines making an angle $\theta, \pi-\theta, \pi+\theta$ and $2\pi -\theta$ with the $x$-axis. 

\section{The space of $4$-ended solutions}

In order to proceed, we need to define precisely the class of entire solutions of (\ref{AC}) we are interested in. As already mentioned, these solutions have the property that their nodal sets are, away from a compact, asymptotic to a finite (even) number of half  oriented affine lines, which are called the {\em ends} of the solutions. The concept of solutions with a finite number of {\em ends} was first introduced in \cite{dkp-2009} and, for the sake of completeness, we recall the precise definitions in the case of $4$-ended solutions.

\medskip

An \emph{oriented affine line} $\lambda \subset \mathbb{R}^{2}$ can be uniquely written as
\[
\lambda:= r \, \mathtt{e}^{\perp}+ \mathbb{R }\, \mathtt{e} ,
\]
for some $r \in\mathbb{R}$ and some unit vector $\mathtt{e} \in S^{1}$, which defines the orientation of $\lambda$. We recall that $\perp$ denotes the rotation by $\pi/2$ in $\mathbb{R}^{2}$. Writing $\mathtt{e} = (\cos\theta, \sin\theta)$, we get the usual coordinates $(r, \theta)$ which allow to identify the set of oriented affine lines with $\mathbb R \times S^1$. 

\medskip

Assume that we are given $4$ oriented affine lines $\lambda_{1}, \ldots, \lambda_4 \subset \mathbb R^2$  which are defined by
\[
\lambda_j := r_j \, \mathtt{e}_j^{\perp}+ \mathbb{R }\, \mathtt{e}_j ,
\]
and assume that these oriented affine lines have corresponding angles $\theta_1, \ldots, \theta_4$ satisfying
\[
\theta_1 < \theta_2 < \theta_3 < \theta_4 < 2\pi + \theta_1 .
\]
In this case, we will say that the $4$ oriented affine lines are {\em ordered} and we will denote by $\Lambda^4_{ord}$ the set of  $4$ oriented affine lines. It is easy to check that for all $R > 0 $ large enough and for all $j=1, \ldots, 4$, there exists $s_{j} \in\mathbb{R}$ such that~:

\begin{itemize}
\item[(i)] The point $ r_{j} \, \mathtt{e}^{\perp}_{j} + s_{j} \, \mathtt{e}_{j}$ belongs to the circle $\partial B_{R}$.\newline

\item[(ii)] The half affine lines
\begin{equation}
\lambda_{j}^{+} := r_{j} \, \mathtt{e}^{\perp}_{j} + s_{j} \, \mathtt{e}_{j} + \mathbb{R}^{+} \, \mathtt{e}_{j} ,
\label{eq:halfline}
\end{equation}
are disjoint and included in $\mathbb{R}^{2} - B_{R}$.\newline

\item[(iii)] The minimum of the distance between two distinct half affine lines $\lambda^{+}_{i}$ and $\lambda_{j}^{+}$ is larger than $4$. 
\end{itemize}

The set of half affine lines $\lambda_{1}^{+}, \ldots, \lambda_{4}^{+}$ together with the circle $\partial B_{R}$ induce a decomposition of $\mathbb{R}^{2}$ into $5$ slightly overlapping connected
components
\[
\mathbb{R}^{2} = \Omega_{0} \cup\Omega_{1} \cup \ldots \cup\Omega_{4} ,
\]
where $\Omega_{0} : = B_{R+1}$ and 
\begin{align}
\label{decomp 1}
\Omega_{j} : = \left( \mathbb R^2 - B_{R-1}\right) Ê\cap \left\{  \mathtt{x} \in\mathbb{R}^{2} \, : \,  \mathrm{dist} (\mathtt{x}, \lambda _{j}^{+}) < \mathrm{dist} (\mathtt{x}, \lambda_{i}^{+})+2, \, \forall i \neq j \,\right\}  ,
\end{align}
for $j=1, \ldots, 4$. Here, $\mathrm{dist} (\cdot , \lambda_{j}^{+})$ denotes the distance to $\lambda_{j}^{+}$. Observe that, for all $j=1, \ldots, 4$, the set $\Omega_{j}$ contains the half affine line $\lambda_{j}^{+}$.

\medskip

Let  ${\mathbb{I}}_{0}, {\mathbb{I}}_{1}, \ldots, {\mathbb{I}}_{4} $ be a smooth partition of unity of $\mathbb{R}^{2}$ which is subordinate to the above decomposition. Hence
\[
\sum_{j=0}^4 \mathbb{I}_{j} \equiv1,
\]
and the support of $\mathbb{I}_{j}$ is included in $\Omega_{j}$. Without loss of generality, we can also assume that ${\mathbb{I}}_{0}\equiv1$ in
\[
\Omega^{\prime}_{0} := B_{R -1},
\]
and ${\mathbb{I}}_{j}\equiv1$ in
\[
\Omega^{\prime}_{j} :=  \left( \mathbb R^2 - B_{R-1}\right) Ê\cap \left\{  \mathtt{x} \in\mathbb{R}^{2} \, : \,  \mathrm{dist} (\mathtt{x}, \lambda_{j}^{+}) < \mathrm{dist} (\mathtt{x}, \lambda_{i}^{+})-2, \, \forall i \neq j \,\right\} ,
\]
for $j=1, \ldots, 4$. Finally, without loss of generality, we can assume that
\[
\|{\mathbb{I}}_{j}\|_{\mathcal{C}^{2}(\mathbb{R}^{2})} \leq C .
\]

With these notations at hand, we define
\begin{align}
u_{\lambda}: = \sum_{j=1}^{4} (-1)^{j} \, {\mathbb{I}}_{j} \, H ( \mathrm{dist}^{s} ( \, \cdot \, , \lambda_{j} )) , \label{def w}
\end{align}
where $\lambda : =  (\lambda_1, \ldots, \lambda_4)$ and 
\begin{equation}
\mathrm{dist}^{s} ( \mathtt{x} , \lambda_{j} ) : = \mathtt{x} \cdot\mathtt{e}^{\perp}_{j} - r_j, \label{eq:signdist}
\end{equation}
denotes the \emph{signed distance} from a point $\mathtt{x} \in\mathbb{R}^{2}$ to $\lambda_{j}$.

\medskip

Observe that, by construction, the function $u_{\lambda}$ is, away from a compact and up to a sign, asymptotic to copies of the heteroclinic solution with ends  $\lambda_{1}, \ldots, \lambda_4$. 

\medskip

Let $\mathcal{S}_{4}$ denote the set of functions $u$ which are defined in $\mathbb{R}^{2}$ and which satisfy
\begin{equation}
u - u_{\lambda}\in W^{2,2} \, (\mathbb{R}^{2}) , 
\label{ass w}
\end{equation}
for some ordered set of oriented affine lines $\lambda_1, \ldots, \lambda_4 \subset \mathbb R^2$. We also define the
\textrm{decomposition operator} $\mathcal{J}$ by
\[
\begin{array}
[c]{rcccllll}
\mathcal{J }: & \mathcal{S}_{4} & \longrightarrow & W^{2,2} (\mathbb{R}^{2})
\times\Lambda^4_{ord} &  &  &  & \\[3mm]
& u & \longmapsto & \left(  u- u_{\lambda}, \lambda\right)  . &  &  &  &
\end{array}
\]
The topology on $\mathcal{S}_{4}$ is the one for which the operator $\mathcal{J}$ is continuous (the target space being endowed with the product topology). 

\medskip
We now have the~:
\begin{definition}
The set  $\mathcal{M}_{4}$ is defined to be the set of solutions $u$ of (\ref{AC}) which belong to $\mathcal{S}_{4}$. 
\label{de:001}
\end{definition}

It is known that $\mathcal{M}_{4}$ is not empty. For example, the {\em saddle solution} constructed in \cite{MR1198672} belongs to $\mathcal M_4$, the nodal set of this solution is the union of the two lines $y= \pm x$. Another important fact, also proven in \cite{MR1198672} or in  \cite{MR2381198}, is that up to a sign and a rigid motion, the {\em saddle solution} is the unique solution whose nodal set coincides with the union of the two lines $y=\pm x$. The solutions constructed in \cite{MR2557944} are also elements of $\mathcal M_4$ and we shall return to this point later on.  

\medskip

Recall from \cite{dkp-2009}, that a solution $u$ of (\ref{AC}) is said to be {\em nondegenerate} if there is no $w \in W^{2,2}(\mathbb R^2)- \{0\}$ which is in the kernel of
\[
L : =  - \Delta + F''(u) ,
\]
and which decays exponentially at infinity.

\medskip

As far as the structure of the set of $4$-ended solutions is concerned, the main result of \cite{dkp-2009} asserts that~:
\begin{theorem} \cite{dkp-2009}
Assume that $u \in \mathcal M_4$ is {\em nondegenerate}, then, close to $u$, $\mathcal M_4$ is a $4$-dimensional smooth manifold.
\label{th:ms}
\end{theorem}
Observe that, given $u \in \mathcal M_4$, translations and rotations of $u$ are also elements of $\mathcal M_4$ and this accounts for $3$ of the $4$ formal dimensions of $\mathcal M_4$, moreover, if $u \in \mathcal M_4$ then  $-u \in \mathcal M_4$.  

\medskip

All the $4$-ended solutions constructed so far have two axis of symmetry and in fact, it follows from a result of C. Gui \cite{2011arXiv1102.4022G} that~:
\begin{theorem} \cite{2011arXiv1102.4022G}
Assume that $u \in{\mathcal{M}}_{4}$. Then, there exists a rigid motion $g$ such that $\bar u : = u \circ g$ is even with respect to the $x$-axis and the $y$-axis, namely
\begin{align}
\label{even}
\bar u(x,y)= \bar u(-x,y)= \bar u(x,-y).
\end{align}
In addition, $\bar u$ is a monotone function of both the $x$ and $y$ variables in the upper right quadrant $Q^{\llcorner}$ defined by
\[
Q^{\llcorner} : = \{ (x,y) \in \mathbb R^2 \, : \,  x>0  \quad  y>0\}, 
\]
and, changing the sign of $\bar u$ if this is necessary, we can assume that
\[
\partial_x \bar u <0 \qquad \mbox{and} \qquad \partial_y \bar u >0 ,
\]
in $Q^{\llcorner}$. 
\end{theorem}

Thanks to this result, we can define the moduli space of $4$-ended solutions by~:
\begin{definition}
The set  $\mathcal{M}_{4}^{even}$ is defined to be the set of $u \in \mathcal{S}_{4}$  which are solutions of (\ref{AC}), are even with respect to the $x$-axis and the $y$-axis and which tend to $+1$ as at infinity along the $y$-axis (and tend to $-1$ at infinity along the $x$-axis). In particular, 
\[
\partial_x u <0 \qquad \mbox{and} \qquad \partial_y u >0 ,
\]
in the upper right quadrant $Q^{\llcorner}$.
\label{de:002}
\end{definition}

\medskip

When studying $\mathcal M^{even}_4$, we restrict our attention to functions which are even with respect to the $x$-axis and the $y$-axis and, in this case, a solution $u \in \mathcal M^{even}_4$ is said to be {\em  even-nondegenerate} if there is no $w \in W^{2,2}(\mathbb R^2)- \{0\}$, which is symmetric with respect to the $x$-axis and the $y$-axis, belongs to the kernel of
\[
L : =  -\Delta + F''(u),
\]
and which decays exponentially at infinity.

\medskip

In the equivariant case (namely solutions which are invariant under both the symmetry with respect to the $x$-axis and the $y$-axis), Theorem~\ref{th:ms} reduces to~:
\begin{theorem}   \cite{dkp-2009}
Assume that $u \in \mathcal M_4^{even}$ is {\em even-nondegenerate}, then, close to $u$, $\mathcal M_4^{even}$ is a $1$-dimensional smooth manifold.
\end{theorem}

\medskip

Any solution $u \in \mathcal M_4^{even}$ has a nodal set which is asymptotic to $4$ half oriented affine lines and, given the symmetries of $u$, these half oriented affine lines are images of each other by the symmetries with respect to the $x$-axis and the $y$-axis. In particular, there is at most one of these half oriented affine line
\[
\lambda : =  r \, {\tt e}^\perp + \mathbb R \, {\tt e} ,
\]
which is included in the upper right quadrant $Q^{\llcorner}$. Writing ${\tt e} = (\cos \theta, \sin \theta)$ where $\theta \in (0, \pi/2)$, we define 
\[
\begin{array}{cccc} 
\mathcal F : & {\mathcal M}_4^{even} & \to & (-\pi/4, \pi / 4) \times \mathbb R,\\[3mm] 
& u & \mapsto & (\theta - \pi/4, r). 
\end{array}
\]
For example, the image by $\mathcal F$ of the {\em saddle solution} defined in \cite{MR1198672} is precisely $(0,0)$, while the images by $\mathcal F$ of the solutions constructed in \cite{MR2557944} correspond to parameters $(\theta, r)$ where $\theta$ is close to $\pm \pi/4$ and $r$ is close to $\mp \infty$.

\medskip

\begin{remark}
Let us observe that, if $u \in \mathcal M_4^{even}$, then $\bar u$ defined by
\[
\bar u(x,y) = - u(y-x),
\]
also belongs to $\mathcal M_4^{even}$ and
\[
\mathcal F (\bar u) = - \mathcal F (u) .
\]
\label{re:1}
\end{remark}

In this paper, we are interested in the understanding of $\mathcal M_4^{even}$. To begin with, we prove that~:

\begin{theorem} [Nondegeneracy]
Any $u \in \mathcal M_4$ is nondegenerate and hence any $u \in \mathcal M_4^{even}$ is even-nondegenerate.
\label{th:1}
\end{theorem}

As a consequence of this result, we find that all connected components of $\mathcal M_4^{even}$ are one-dimensional smooth manifolds. Moreover, as a byproduct of the proof of this result, we  also obtain that the image by $\mathcal F$ of any connected component of $\mathcal M_4^{even}$ is a smooth immersed curve in $(-\pi/4, \pi/4)\times \mathbb R$. Thanks to Remark~\ref{re:1}, we find that the image of any connected component of $\mathcal M_4^{even}$ by $\mathcal F$ is invariant under the action of the symmetry with respect to $(0,0)$.

\medskip

To proceed, we define the {\em classifying map} to be the projection of $\mathcal F$ onto the first variable
\[
\begin{array}{cccc} 
\mathcal P : & {\mathcal M}_4^{even} &      \to       & (-\pi/4, \pi / 4) ,\\[3mm] 
                      &                   u                      & \mapsto & \theta - \pi/4. 
\end{array}
\]
Our second result reads~:
\begin{theorem} [Properness]
The mapping $\mathcal P$ is proper, i.e. the pre-image of a compact in $(-\pi/4, \pi/4)$ is compact in $\mathcal M^{even}_4$ (endowed with the topology induced by the one of $\mathcal S_4$). 
\end{theorem}
%
%
\medskip

The solutions with almost parallel ends constructed in \cite{MR2557944} belong to one of the connected component of $\mathcal M_4^{even}$ and we also know that the saddle solution also belongs to a connected component of $\mathcal M_4^{even}$. In principle, it could be possible that $\mathcal M_4^{even}$ contained many different connected  components and it could also be possible that $\mathcal M_4^{even}$ contained connected components which are diffeomorphic to $S^1$. Nevertheless, we prove that~:
\begin{theorem}
All connected components of $\mathcal M^{even}_4$ are diffeomorphic to $\mathbb R$, i.e. there is no closed loop in $\mathcal M_4^{even}$.
\end{theorem}

Looking at the image by $\mathcal P$ of the connected component of $\mathcal M_4^{even}$ which contains  the saddle solution, we conclude from the above results that~:
\begin{theorem} [Surjectivity of $\mathcal P$]
The mapping $\mathcal P$ is onto.   
\label{th:4}
\end{theorem}

As a consequence, for any $\theta \in (0, \pi/2)$, there exists a solution $u \in \mathcal M_4^{even}$ whose nodal set at infinity is asymptotic to the half oriented affine lines whose angles with the $x$-axis are given by $\theta, \pi-\theta, \pi+\theta$ and $2\pi -\theta$.

\medskip

Given all the evidence we have, it is tempting to conjecture that $\mathcal M_4^{even}$ has only one connected component and that the image of $\mathcal M_4^{even}$ by $\mathcal F$ is a smooth embedded curve. Moreover, it is very likely that $\mathcal P$ is a diffeomorphism from $\mathcal M_4^{even}$ onto $(-\pi/4, \pi/4)$. Observe that Theorem~\ref{th:4} already proves that $\mathcal P$ is onto. 

\medskip

To give credit to the above conjecture, in \cite{partII}, we will show that  $\mathcal M_4^{even}$ has only one connected component which contains both the saddle solution and the solutions constructed in \cite{MR2557944}. The proof of this last result is rather technical and uses tools which are different from the one needed to prove the results in the present paper and this is the reason why, we have chosen to present it in a separate paper \cite{partII}. 

\medskip

To complete this list of results, we mention an interesting by-product of the proof of Theorem~\ref{th:1}. Assume that $u$ is a solution of (\ref{AC}) and denote by 
\[
L := - \Delta + F''(u),
\] 
the linearized operator about $L$. Recall that, if $\Omega$ is a bounded domain in $\mathbb R^2$,  then the index of $L$ in $\Omega$ is given by the number of negative eigenvalues of the operator $L$ which belong to $W^{1,2}_0(\Omega)$. Following \cite{MR808112}, we have the~:
\begin{definition}
The function $u$, solution of (\ref{AC}), has finite Morse index if the index of every bounded domain $\Omega \subset \mathbb R^2$ has a uniform upper bound. 
\end{definition}

\medskip

And in this paper, we prove the~:
\begin{theorem}[Morse index]
Any $2k$-ended solution of (\ref{AC}) has finite Morse index. 
\label{th:2.8}
\end{theorem}
We will only prove this result for $4$-ended solutions but the proof extends {\it verbatim} to any $2k$-ended solution. 

\medskip

Since the Morse index of a $2k$-ended solution $u$ is finite (equal to $m$), we know from \cite{MR808112}, that there exists a finite dimensional subspace $E \subset L^{2}(\mathbb{R}^{2})$, with $\mathrm{dim}\,E =m$, which is spanned by the eigenfunctions $\phi_{1}, \ldots,\phi_{m}$ of the operator $L$, corresponding to the negative eigenvalues $\mu_{1}, \ldots,\mu_{m}$ of $L$.

\medskip

We now sketch the plan of our paper.

\medskip

In section 3, we prove that {\em any} element of ${\mathcal{M}}_{4}^{even}$ is even-nondegenerate (we also prove that it is nondegenerate, even though we do not need this result). The proof follows the line of the proof in \cite{kow-liu} where it is proven that the saddle solution is nondegenerate. This will prove Theorem~\ref{th:1} and, thanks to this result, it will then follow from the Implicit Function Theorem (see Section 8 and Theorem 2.2 in \cite{dkp-2009}) that any connected component of $\mathcal M_4^{even}$ is $1$-dimensional.

\medskip

In section 4, we recall two key tools which will be needed in the analysis of the properness of the classifying map $\mathcal P$. The first is a well known {\it a priori} estimate for solutions of (\ref{AC}) which states that, away from its zero set, the solutions of (\ref{AC}) tend to $\pm 1$ exponentially fast. The second tool is a {\em balancing formula} which holds for any solution of (\ref{AC}). This balancing formula reflects the invariance of our problem under translations and rotations and can be understood as a consequence of Noether's Theorem. 

\medskip

In section 5, we prove the properness of the classifying map $\mathcal P$. Assume that $(u_n)_{n\geq 0}$ is a sequence of solutions of ${\mathcal{M}}^{even}_{4}$ such that $\mathcal P (u_n)$ remains bounded away from $-\pi/4$ and from $\pi/4$, further assume that $(u_n)_{n\geq 0}$ converges on compacts to $u$ (thanks to elliptic estimates, this can always be achieved up to the extraction of a subsequence). We will show that $u \in \mathcal M_4^{even}$ and also that 
\[
\mathcal P (u) =\lim_{n\to \infty} \mathcal P (u_n).
\]
The key tool in the proof is the use of the balancing formula introduced in the previous section which allows one to control the nodal sets of $u_n$ as $n$ tends to infinity. As we will see, this compactness result implies that the image by $\mathcal P$ of the connected component of $\mathcal M_4^{even}$ which contains the saddle solution, is the entire interval $(-\pi/4,\pi/4)$. 

\medskip

Section 6 is devoted to the proof of the non existence of compact components in ${\mathcal{M}}_{4}^{even}$. We will show that a connected component in ${\mathcal{M}}_{4}^{even}$ cannot be compact (i.e. cannot be diffeomorphic to $S^1$). As a consequence, this will imply that the image of {\em any} connected components of $\mathcal M_4^{even}$ by $\mathcal P$ is the entire interval $(-\pi/4, \pi/4)$.

\medskip

In section 7, we prove that any $2k$-ended solution of (\ref{AC}) has finite Morse index. The proof relies on an intermediate result used in section 4, in the proof of the nondegeneracy of the $4$-ended solutions of (\ref{AC}).

\medskip

Our  results are very much inspired from a similar classification result which was obtained in a very different framework~: the theory of minimal surfaces. Let us briefly explain the analogy between our result and the corresponding result in the theory minimal surfaces in $\mathbb{R}^{3}$. 

\medskip

In 1834, H.F. Scherk discovered an example of a singly-periodic, embedded, minimal surface in $\mathbb{R}^{3}$ which, in a complement of a vertical cylinder, is asymptotic to $4$ half planes with angle $\pi/2$ between them (these planes are called ends). This surface, after an appropriate rigid motion and scaling, has two planes of symmetry, say the $x_2=0$ plane and the $x_{1}=0$ plane, and it is periodic, with period $2\pi$ in the $x_{3}$ direction. If $\theta \in (0, \pi/2)$ denotes the angle between the asymptotic ends of the Scherk's surface contained in $\{(x_1, x_2, x_3) \in \mathbb R^3 \, : \, x_{1}>0, \quad x_{2}>0\}$ and the $x_{2}=0$ plane, then for the original Scherk surface corresponds to $\theta = \pi/4$. This surface is the so called {\em Scherk's second surface} and it will denoted here by $S_{\pi/4}$. 

\medskip

In 1988, H. Karcher \cite{MR958255} found a one parameter family of Scherk's type surfaces with $4$-ends including the original example. These minimal surfaces are parameterized by the angle $\theta\in(0, \pi/2)$ between one of their asymptotic planes and the $x_{2}=0$ plane. The one parameter family $(S_{\theta} )_{\theta \in (0, \pi/2)}$ of these
surfaces, normalized in such a way that the period in the $x_{3}$ direction is $2\pi$, is the family of Scherk singly periodic minimal surfaces. 

\medskip

We note that the $4$-ended elements of Scherk family are given explicitly in terms of the Weierstrass representation, or alternatively they can be represented implicitly as the solutions of
\begin{align*}
\cos^{2}\theta \, \cosh\left( \frac{x_{1}}{\cos \theta}\right) - \sin^{2}\theta \, \cosh\left( \frac{x_{2}}{\cos\theta}\right) = \cos x_{3}.
\end{align*}
More generally, Scherk's surfaces with $2k$-ends have also been constructed by H. Karcher \cite{MR958255}. They have been classified by J. Perez and M. Traizet in \cite{Per-Tra}. In some sense our result can be understood as an analog of the classification result of J. Perez and M. Traizet for $4$-ended Scherk's surfaces. 

\section{The nondegeneracy of $4$-ended solutions}
\label{sec no degeneracy}

In this section, we prove that any  $u \in \mathcal M^{even}_4$ is even-nondegenerate. The proof follows essentially the proof of the nondegeneracy of the saddle solution in \cite{kow-liu} and subsequently this idea was used by X. Cabr\'e in \cite{Cab}. The main result is the~:
\begin{theorem}
Assume that $u \in \mathcal M_4^{even}$ and $\delta >0$. Further assume that $\varphi \in e^{-\delta (1+|{\tt x}|^2)}  \, W^{2,2} (  \mathbb{R}^{2} )$ is a solution of 
\[
\left( \Delta - F^{\prime\prime}\left(  u\right) \right) \varphi=0,
\]
in $\mathbb R^2$ which is symmetric with respect to both the $x$-axis and the $y$-axis, then $\varphi \equiv 0$.
\label{th:ng}
\end{theorem}

As in \cite{kow-liu}, the  proof of this Proposition relies on the construction of a supersolution for the operator $L$, away from a compact. To explain the main idea of the proof, let us digress slightly and consider the heteroclinic solution $(x,y) \mapsto H(x)$ and define 
\[
L_0: =  -\Delta + F''(H'),
\] 
the linearized operator about the heteroclinic solution. Clearly, the function 
\[
\Psi_0(x,y): =  H'(x),
\] 
is positive and is a solution of $L_0 \, \Psi_0 =0$. Since any $u \in \mathcal M^{even}_4$ is asymptotic to a heteroclinic solution, we can transplant $H'$ along the ends of $u$ to build a positive supersolution for $L : = -\Delta + F'(u)$. More precisely, we have the~:
\begin{proposition} 
\label{pr:3.1}
Under the above assumptions, there exist $R_0 > 0$ and a function $\Psi >0$ defined in $\mathbb R^2$ such that 
\[
\left( \Delta -F^{\prime\prime}\left(  u\right) \right) \, \Psi \leq 0,
\]
in $\mathbb R^2 - B(0, R_0)$. 
\end{proposition}
\begin{proof}
The proof follows from a direct construction of the function $\Psi$. In the upper right quadrant $Q^{\llcorner}$, the zero set of $u$ is asymptotic to the half of an oriented affine line 
\[
\lambda = r \, {\tt e}^\perp + \mathbb R \, {\tt e},
\]
with ${\tt e } : =  (\cos \theta, \sin \theta)$. Without loss of generality, we can assume that $\theta \in [\pi/4, \pi/2)$ since, if this is not the case, we just compose $-u$ with a rotation by $\pi/2$. The intersection of $\lambda$ with the closure of the upper right quadrant, $Q^{\llcorner}$ will be denoted by 
\[
\bar \lambda_1^+ : = \overline{Q^{\llcorner}} \cap \lambda, 
\]
and its image by the symmetry with respect to the $y$-axis will be denoted by $\bar \lambda_2^+$, while its image by the symmetry with respect to the $x$-axis will be denoted by $\bar \lambda_4^+$. Finally, the image of $\bar \lambda_2^+$ by the symmetry with respect to the $x$-axis is equal to the the image of $\bar \lambda_4^+$ by the symmetry with respect to the $y$-axis and will be denoted by $\bar \lambda_3^+$. So, at infinity, the zero set of $u$ is asymptotic to $\bar \lambda^+_1, \ldots , \bar \lambda_4^+$. We denote by $q_j^+$ the end points of $\bar \lambda_j^+$, namely 
\[
\{q_j^+\} :=  \partial  \bar \lambda_j^+ .
\] 
Observe that $\{q_j^+ \, : \, j=1, \ldots, 4\}$ contains at most $2$ points and we denote by $\Gamma$ the line segment joining these two points. Also observe that $\mathbb R^2 - (\bar \lambda_2^+ \cup \Gamma \cup \bar \lambda_4^+)$ has two connected components, one of which contains $\bar \lambda_1^+$ and will be denoted by $U_1$ while the other, which contains $\bar \lambda_3^+$, will be denoted by $U_3$.

\medskip

The crucial observation is the following : If 
\[
f (x,y) : = (1- e^{-\mu y} ) \, H'(x),
\]
then, using the fact that $ (\Delta - F''(H))  \, H'=0$, we get 
\[
(\Delta - F''(H)) \, f = - \mu^2 \, e^{-\mu y} \, H' < 0 .
\]

We define the function $h$ by
\[
h ((r+s) \, {\tt e}^\perp + t \, {\tt e}) =  (1- e^{-\mu t} ) \, H'(s) .
\]
Observe that the function $h$ is defined in all $\mathbb R^2$. Nevertheless, since we are only interested in this function in  $U_1$, we define a smooth cutoff function $\chi$ which is identically equal to $1$ in $U_1$ at distance $1$ from  $\partial U_1$ and which is identically equal to $0$ in $U_3 : = \mathbb R^2 - \overline{U_1}$. As usual, we assume that $| \nabla \chi |\leq C$, for some $C >0$,  as we are entitled to do.

\medskip

Using $\chi$ and $h$, we build our supersolution in such a way that it is invariant under the symmetry with respect to the $x$-axis and under the symmetry with respect to the $y$-axis. We define
\[
\begin{array}{rllll}
\Psi (x,y) & : = &  \chi (x,y) \, h (x,y) +  \chi (-x,-y) \, h (-x,-y)  \\[3mm]
                &  + & \chi (x,-y) \, h (x,-y) +  \chi (-x, y) \, h (-x, y).
\end{array}
\]

We know from the Refined Asymptotics Theorem (Theorem 2.1 in \cite{dkp-2009}) that, as $t$ tends to infinity, $(r,s) \mapsto u ((r+s) \, {\tt e}^\perp + t \, {\tt e})$ converges exponentially fast to $(s,t) \mapsto H(s)$ uniformly in $s \in [-\rho, \rho]$. Using this property, we see that we can chose $\mu >0$ close enough to $0$ such that
\begin{equation}
L \Psi  < - \frac{\mu^2}{2} \, e^{-\mu y} \, H'
\label{es1}
\end{equation}
in a tubular neighborhood of width $\rho$ around $\partial U_1$ and away from a ball of radius $R_0$ large enough, centered at the origin. Observe that the choice of $\mu$ only depends on the decay of $(s,t) \mapsto u((r+s) \, {\tt e}^\perp + t \, {\tt e})$ towards $(s,t) \mapsto H(s)$ and does not depend on the choice of $\rho$. However, increasing $\rho$ affects the minimal value of $R_0$ for which (\ref{es1}) holds.  

\medskip

Now, we choose $\rho >0$ and $R_0>0$ large enough, so that 
\[
L \Psi < - \frac{\mu^2}{2} \, e^{-\mu y} \, H', 
\]
away from $B_{R_0}$ and away from a tubular neighborhood of width $\rho$ around $\bar \lambda_1\cup \ldots \cup \bar  \lambda^+_4$. Here, we simply use the fact that $F''(u)$ converges uniformly to $F''(\pm 1)$ away from the nodal set of $u$.  This completes the proof of result.
\end{proof}

\medskip

Observe that this construction is not specific to the case of $4$-ended solutions of (\ref{AC}) and in fact a similar construction would hold for any $2k$-ended solution. With this Lemma at hand, we can adapt the argument in \cite{kow-liu} where the nondegeneracy of the saddle solution is proven and we simply adapt it to the general case where $u$ is any $4$-ended solution.

\medskip

\begin{proof}[Proof of Theorem~\ref{th:ng}] Recall that, by definition, since $u \in \mathcal M_4^{even}$, we have
\[
\partial_y u >0 \qquad \mbox{when} \qquad y>0 ,
\]
and 
\[
\partial_x u < 0 \qquad \mbox{when} \qquad x >0 .
\]
Let $\varphi$ be the function as in the statement of Theorem~\ref{th:ng}. It follows from the Linear Decomposition Lemma (Lemma 4.2 in \cite{dkp-2009}) that $\varphi$ decays exponentially at infinity. More precisely, there exist constants $\alpha , C >0$ such that 
\[
| \varphi ({\tt x}) | \leq  C \, e^{-\alpha |{\tt x}|} ,
\]
for all ${\tt x} \in \mathbb R^2-B(0,1)$.  

\medskip

{\bf Step 1}. We assume that $\varphi$ is not identically equal to $0$ and define $\mathcal Z$ to be the zero set of  $\varphi$. Since $\varphi$ is assumed to be symmetric with respect to both the $x$-axis and the $y$-axis, so is the set $\mathcal Z$.  Clearly Proposition~\ref{pr:3.1} together with the maximum principle, implies that $\mathbb R^2 - \mathcal Z$ has no bounded connected component included in $\mathbb R^2 - B(0, R_0)$, since $\Psi$ can be used as a supersolution to get a contradiction.

\medskip

{\bf Step 2}.  We claim that any unbounded connected component of $\mathbb R^2 - \mathcal Z$ necessarily contains $\mathbb R^2 - B(0, R)$ for some $R$ large enough. Indeed, if this were not the case then, using the symmetries of $\varphi$, one could find $\Omega \subset \mathbb R^2$, an unbounded connected component of $\mathbb R^2 -\mathcal Z$, which is included in one of the four half spaces $\{ (x,y) \in \mathbb R^2 \, : \, \pm x >0\}$ or $\{ (x,y) \in \mathbb R^2 \, : \, \pm y >0\}$. For example, let us assume that 
\[
\Omega \subset \{ (x,y) \in \mathbb R^2 \, : \, y >0\}.
\] 
Following \cite{kow-liu}, we  adapt the proof of the de Giorgi conjecture in dimension $2$ by N. Ghoussoub and C. Gui to derive a contradiction. 

\medskip

We define $\psi : = \partial_y u$ which is a solution of $(\Delta - F''(u)) \, \psi =0$ in $\mathbb R^2$. Moreover, $\psi >0$ in $\{ (x,y) \in \mathbb R^2 \, : \, y >0\}$ and we check from direct computation that
\begin{equation}
\mbox{div} \, \left( \psi^2 \nabla h \right) = 0 ,
\label{eq:3.1}
\end{equation}
where $h : = \frac{ \varphi}{\psi}$. For all $R \geq 1$, we consider a cutoff function $\zeta_R$ which is identically equal to $1$ in $B(0,R)$, identically equal to $0$ outside $\mathbb R^2 - B(0,2R)$ and which satisfies $|\nabla \zeta_R| \leq C/R$ for some constant $C>0$ independent of $R\geq 1$.

\medskip

We multiply (\ref{eq:3.1}) by $\zeta^2_R \, \psi$ and integrate the result over $\Omega$. We find after an integration by parts
\[
\int_\Omega |\nabla h|^2 \, \psi^2 \, \zeta^2_R \, d{\tt x} + 2 \, \int_\Omega \psi^2 \, h \, \zeta_R \, \nabla h \, \nabla \zeta_R \, d{\tt x} =0 .
\]
Observe that, in the integration by parts, some care is needed when the boundary of $\Omega$ touches the $x$-axis but it is not hard to see that the integration by parts is also legitimate in this case (we refer to \cite{kow-liu} for details). Then, Cauchy-Schwarz inequality yields
\begin{equation}
\int_\Omega |\nabla h|^2 \psi^2 \zeta_R^2 d{\tt x} \leq 2 \left( \int_{\Omega \cap A_R} |\nabla h|^2 \psi^2  \zeta^2_R d{\tt x}\right)^{1/2} \left( \int_{\Omega \cap A_R} \varphi^2 |\nabla \zeta_R|^2 d{\tt x}\right)^{1/2} ,
\label{eq:45}
\end{equation}
where $A_R : =  B (0, 2R) - \overline B(0, R)$ contains the support of $\nabla \zeta_R$.  Hence, 
\[
\int_\Omega |\nabla h|^2 \, \psi^2 \, \zeta^2_R \, d{\tt x} \leq 4 \, \left( \sup_{A_R} | \varphi |^2 \right) \, \int_{\Omega \cap A_R}  |\nabla \zeta_R|^2 \, d{\tt x} .
\]
By construction of $\zeta_R$, the integral on the right hand side is bounded independently of $R$ and since $\varphi \in W^{2,2}(\mathbb R^2)$ we know that $\varphi$ tends to $0$ at infinity. Letting $R$ tend to infinity, we conclude that 
\[
\int_\Omega |\nabla h|^2 \, \psi^2 \, d{\tt x} =0 ,
\]
which then implies that $h\equiv 0$ in $\Omega$ and hence we also have $\varphi \equiv 0$  in this set. Finally, $\varphi \equiv 0$ in $\mathbb R^2$ by the unique continuation theorem. This is certainly a contradiction and the proof of the claim is complete. 

\medskip

{\bf Step 3}.  By the above, we know that $\varphi$ does not change sign away from a compact and, without loss of generality, we can assume that $\varphi >0$ in $\mathbb R^2 - B(0,R)$ for some $R>0$ large enough. As in \cite{kow-liu}, we proceed by analyzing the projection of $\varphi$ onto $H^{\prime}$ (composed with a suitable rotation).  

\medskip

The nodal set of $u$ in the upper right quadrant $Q^{\llcorner}$ is asymptotic to an oriented half line which is denoted by $\lambda$ and is given by 
\[
\lambda  = r\, {\tt e}^\perp + \mathbb R \, {\tt e} ,
\]
where ${\tt e} = (\cos \theta, \sin \theta)$. Up to a rotation by $\pi/2$ and a possible change of sign, we can assume that $\theta \in [\pi/4, \pi/2)$.  The image of $\lambda$ through the symmetry with respect to the $y$-axis will be denoted by $\bar \lambda$. Notice that, since $u$ is symmetric with respect to the $y$-axis, $\bar \lambda$ is also asymptotic to the zero set of  $u$.  

\medskip

Given the expression of $\lambda$, we define 
\[
\tilde u (s,t) : =  u  ( (r+s)\, {\tt e}^\perp + t \, {\tt e} ) \qquad \mbox{and} \qquad  \tilde \varphi (s,t) : =  \varphi ( (r+s)\, {\tt e}^\perp + t \, {\tt e}) .
\]
We consider the function $g$ defined by
\[
g (t) := \int_{\mathbb R} \tilde \varphi ( s,t ) \, H^{\prime} (s) \, ds .
\]
Since $\varphi >0$ away from a compact, we conclude that $g\geq 0$ for $t >0$ large enough. We have, using the equation satisfied by $\varphi$ 
\[
g'' (t) := \int_{\mathbb R} \partial_t^2 \tilde \varphi (s,t) \, H^{\prime} (s) \, ds =  - \int_{\mathbb R} (\partial_s^2  - F''(\tilde u (s,t))) \, \tilde \varphi (s,t) \, H^{\prime} (s) \, ds , 
\]
and an integration by parts yields
\[
g'' (t) := - \int_{\mathbb R} \tilde \varphi (s,t) \, \partial_s^2 H^{\prime} (s) \, ds  + \int_{\mathbb R}  F''(\tilde u(s,t))) \, \tilde \varphi (s,t) \, H^{\prime} (s) \, ds. 
\]
Finally, using the equation satisfied by $H'$, we conclude that
\[
g^{\prime\prime} (t) =\int_{\mathbb R} \left( F^{\prime\prime} (\tilde u(s,t)) -F^{\prime\prime }\left(  H (s)  \right)\right) \,  \tilde \varphi (s,t) \,  H^{\prime} (s) \, ds. 
\]
Observe that $\varphi$ tends exponentially to $0$ at infinity and hence so does $g$. Integrating the above equation from $t$ to $\infty$, we conclude that 
\[
g^{\prime} (t) = - \int_{t}^\infty \int_{\mathbb R} \left( F^{\prime\prime} (\tilde u(s,z)) -F^{\prime\prime
}\left(  H (s)  \right)\right) \,  \tilde \varphi (s,z) \,  H^{\prime} (s) \, ds \, dz. 
\]

Recall that $\bar \lambda$ is the image of $\lambda$ through the symmetry with respect to the $y$-axis and that we can parameterize $\bar \lambda$ by 
\[
\bar \lambda  =  \bar r \, \bar {\tt e}^\perp + \mathbb R \, \bar {\tt e},
\]
with obvious relations between ${\tt e}$ and $\bar {\tt e}$. We define
\[
\bar g (t) := \int_{\mathbb R} \varphi (  (\bar r+s)\, \bar {\tt e}^\perp + t \, \bar {\tt e} ) \, H^{\prime} (s) \, ds .
\]
Observe that, by symmetry of both $u$ and $\varphi$, we have 
\[
\bar g(t) = g(t) ,
\]
for all $t \geq 0$. 

\medskip

We claim that there exist constants $C >0$ and $\beta >0$ such that 
\begin{equation}
\left\vert g^{\prime}\left( t \right)  \right\vert \leq C\, e^{-\beta t} \,  \int_{t}^{+\infty} g\left( z \right)  \, dz. 
\label{gdouble}
\end{equation}
for all $t>0$ large enough. Assuming that we have already proven this inequality, it is then a simple exercise to check that the only solution to this differential inequality is identically equal to $0$. Hence 
\[
\int_{\mathbb R} \tilde \varphi (s, t) \, H^{\prime} (s) \, ds =0 ,
\]
for all $t>0$ large enough. Since the integrand is non negative this implies that $\tilde \varphi (s,t) \equiv 0$ for all $s\in \mathbb R$ and all $ t >0$ large enough. Therefore, $\varphi \equiv 0$ by the unique continuation theorem. This is again a contradiction and hence this completes the proof of the Theorem.

\medskip

It remains to prove (\ref{gdouble}). Observe that, in the definition of $g$, the domain of integration contains both a half of $\lambda$ and $\bar \lambda$ (this is where we use the fact that $\theta \in [\pi/4, \pi/2)$). We use the fact  that, thanks to  the Refined Asymptotics Theorem (Theorem 2.1 in \cite{dkp-2009}), $u$ is exponentially close to the sum of the heteroclinic solution $H'$ along $\lambda$ and also along $\bar \lambda$. Close to $\lambda$, we can therefore estimate 
\[
|F^{\prime\prime} (\tilde u(s,t)) -F^{\prime\prime}\left(  H (s)  \right)|\leq C \, e^{-\beta t} ,
\]
for some $\beta >0$. In fact this estimate holds at any point of $\{ (x,y)\in \mathbb R^2 \, : \, y > 0\}$ which is closer to $\lambda$ than to $\bar \lambda$. 

\medskip

We can write any point $(r+s) \, {\tt e}^\perp + t\, {\tt e}$ close to $\bar \lambda$ as  $(\bar r+\bar s) \, \bar{\tt e}^\perp + \bar t\, \bar {\tt e}$. Therefore, at any such point which is closer to $\bar \lambda$ than to $\lambda$, we simply use the fact that 
\[
| F^{\prime\prime} (u((r+s) \, {\tt e}^\perp + t\, {\tt e})) -F^{\prime\prime}\left(  H (s))  \right) \, H'(s) | \leq C \, e^{-\beta \bar t} \, H' (\bar s)  ,
\]
and we conclude that 
\[
\begin{array}{rllll}
\displaystyle \left| \int_{t}^\infty \int_{\mathbb R} \left( F^{\prime\prime} (\tilde u(s,z)) -F^{\prime\prime
}\left(  H (s)  \right)\right) \,  \tilde \varphi (s,z) \,  H^{\prime} (s) \, ds \, dz \right| \qquad \qquad\\[3mm]
\displaystyle \leq C \, e^{-\beta t} \,\int_{t}^\infty  \left( g(z) + \bar g(z)\right)\, dz. 
\end{array}
\]
Finally, the estimate (\ref{gdouble}) follows from the fact that $\bar g =g$. 
\end{proof}

\medskip

We now explain how to prove that any $u \in \mathcal M^{even}_4$ is non-degenerate. If $\phi \in e^{-\delta (1+|{\tt x}|^2)} \, W^{2,2}(\mathbb R^2)$ is a solution of $(\Delta - F'(u))\, \phi =0$, we can decompose $\phi = \phi_{e} + \phi_o$ into the sum of two functions, one of which $\phi_e$ being even under the action of the symmetry with respect to the $x$-axis and the other one $\phi_o$ being odd under the action of the same symmetry. Since $\phi_o$ vanishes on the $x$-axis, we can use the argument already used in Step 2 to prove that $\phi_o\equiv 0$ and hence $\phi$ is even under the action of the symmetry with respect to the $x$-axis. Using similar arguments one also prove that  $\phi$ is even under the action of the symmetry with respect to the $y$-axis and the non-degeneracy follows from Theorem~\ref{th:ng}. 

\medskip

Thanks to Theorem~\ref{th:ng}, we can apply the Implicit Function Theorem (see Section 8 and Theorem 2.2 in \cite{dkp-2009}) to show that any connected component of $\mathcal M_4$ is $4$-dimensional and, equivalently, we conclude that any connected component of $\mathcal M_4^{even}$ is $1$-dimensional  (the rational being that, the formal dimension of the moduli space of solutions of (\ref{AC}) is equal to the number of ends but, because of the symmetries, elements of $\mathcal M^{even}_4$ have only one end in the quotient space). Moreover, a consequence of this Implicit Function Theorem is that close to $u \in \mathcal M^{even}_4$, the space $\mathcal M^{even}_4$ can either be parameterized by the angle $\alpha$ or by the distance $r$ and this implies that the image of {\em any} connected component of $\mathcal M_4^{even}$ by the mapping $\mathcal F$ is an immersed curve. 

\section{Two useful tools}

\subsection{An {\it a priori} estimate}

It is well known that any solution $u$ of (\ref{AC}) which satisfies $\left\vert u\right\vert <1$ tends to $\pm 1$ exponentially fast away from its nodal set. In particular, we have the~:
\begin{lemma}
\label{le:4.1}
Given $\delta \in (0,1)$, there exists $\rho_\delta >0$ such that, for any solution of (\ref{AC}) which satisfies  $\left\vert u\right\vert <1$, we have
\begin{equation}
B(\mathtt{x}, 2 \rho_\delta) \subset \mathbb R^2 - \mathcal Z (u) \quad \Rightarrow  \quad |u^2 - 1| \leq \delta \quad \mbox{in} \quad B(\mathtt{x}, \rho_\delta) , 
\label{eq:est-2}
\end{equation}
where
$$
\mathcal Z (u) :=\{ {\tt x} \in \mathbb R^2 \, : \, u({\tt x}) =0\}, 
$$ 
denotes the nodal set of the function $u$.
\end{lemma}

This result is a simple corollary of  the result below, whose proof can already be found in \cite{MR1470317} (see Lemma 3.1-Lemma 3.3 therein) and also in \cite{MR1637919}~:
\begin{lemma}
\label{le:2.2}
There exist constants $C>0$ and $\alpha >0$ such that, for any solution of (\ref{AC}) which satisfies  $\left\vert u\right\vert <1$, we have
\begin{equation}
|u(\mathtt{x})^{2} - 1| + |\nabla u(\mathtt{x})| + |\nabla^{2}u(\mathtt{x})| \leq C \, e^{-\alpha {\mathrm{dist}}(\mathtt{x},\mathcal Z(u))} , \label{exp est 1}
\end{equation}
for all ${\tt x} \in \mathbb R^2$. 
\end{lemma}
\begin{proof} Since this Lemma plays a central role in our result, we give here a complete proof for the sake of completeness. 

\medskip

We denote by $\phi_R$ the eigenfunction which is associated to the first eigenvalue of  $-\Delta$ on the ball of radius $R$, under $0$ Dirichlet boundary conditions. We assume that $\phi_R$ is normalized so that $\phi_R(0) = \sup_{B(0,R)} \phi_R =1$. Recall that  the associated eigenvalue $\mu_R$ satisfies $\mu_R = \mu_1 /R^2$. 

\medskip

Given $\delta \in (0,1)$, we choose $R_0 >0$ such that 
\[
 - F^{\prime} \left(  s\right) \, R_0^2 \geq  \mu_{1} \, s, 
\]
for all $s \in [0, 1-\delta]$. Assume that $R >R_0$ and that $B({\tt x}, 2R) \subset \mathbb R^2 - \mathcal Z (u)$. To simplify the discussion, let us also assume that $u >0$ in $B({\tt x}, 2R)$. 

\medskip

We claim that $u \geq 1-\delta$ in $B({\tt x}, R)$.  Indeed, if this is not the case then there exists $\bar {\tt x} \in B({\tt x}, R)$ such that $u(\bar {\tt x}) < 1-\delta$. In this case, we define $\epsilon >0$ to be the largest positive real such that 
\[
u \geq \epsilon \, \phi_R (\cdot -\bar {\tt x}),
\]
in $B (\bar {\tt x}, R)$. Certainly $\epsilon \leq 1-\delta$ and there exists ${\tt z} \in B (\bar {\tt x}, R)$ such that 
\[
u ({\tt z}) = \epsilon \, \phi_R ({\tt z} -\bar {\tt x}) \leq 1-\delta.
\]
By construction of $R$, we can write
\[
- \epsilon \, \Delta \phi_R = \frac{\mu_1}{R^2} \,  \epsilon \, \phi_R < F'(\epsilon \phi_R) ;
\]
Since $-\Delta u = - F'(u)$ we conclude that 
\[
- \Delta (\epsilon \phi_R - u) < 0 ,
\]
at the point ${\tt z}$ and this contradicts the fact that $u-\epsilon \phi_R$ has a local minimum at ${\tt z}$. The proof of the claim is complete. 

\medskip

We now fix $\alpha >0$ such that $\alpha^2 < F''(1)$ and we choose $\delta \in (0,1)$ close to $1$ so that $F''(t) \geq \alpha^2$ for all $t \in [1-\delta , 1]$. According to the above claim, we know that for all $R > R_0$, $u \geq 1-\delta$ (or $u < \delta -1$) in $B(\bar {\tt x}, R)$ provided $B(\bar {\tt x}, 2R) \subset \mathbb R^2 - \mathcal Z (u)$. Therefore, we get
\[
- \Delta (1-u) = -  \frac{F'(u) - F'(1)}{u-1} \, (1-u) \leq - \alpha^2 \, (1-u) ,
\]
in $B(\bar {\tt x}, R)$. A direct computation shows that
\[
(-\Delta + \alpha^2) \, e^{-\alpha \sqrt{1+ r^2}} \geq 0 ,
\]
where $r := | {\tt x} - \bar{\tt x}|$. This, together with the maximum principle, implies the exponential decay of $1- u^2$ away from $\mathcal Z(u)$, the zero set of $u$.
\end{proof}

\subsection{The balancing formul\ae}

We describe the balancing formul\ae\  for solutions of (\ref{AC}) in the form they were introduced in \cite{dkp-2009}.  Assume that $u$ is a smooth function defined in $\mathbb{R}^{2}$ and $X$ a vector field also defined in $\mathbb R^2$. We define the vector field
\[
\Xi (X,u) : =   \left( \frac12 | \nabla u |^{2} + F(u) \right)  X - X(u) \nabla u .
\]

Recall that Killing vector fields are vector fields which   generate   the group of isometries of $\mathbb R^2$.  They are linear combinations (with constant coefficients) of the constant vector fields $\partial_x$ and $\partial_y$ generating the group of translations and the vector field $x \, \partial_{y} - y \, \partial_{x}$ which generates the group of rotations in $\mathbb R^2$. 

\medskip

We have the~:
\begin{lemma}[Balancing formul\ae] \cite{dkp-2009}   Assume that $u$ is a solution of (\ref{AC}) and that $X$ is a Killing vector field. Then  $\mbox{\rm div} \,  \Xi (X,u)  = 0$.
\label{le:PI}
\end{lemma}
\begin{proof} To prove this formula, just multiply the equation (\ref{AC}) by $X(u)$ and use simple manipulations on partial derivatives. \end{proof}

This result is nothing but an expression  of the invariance of (\ref{AC}) under the action of rigid motions. To see how useful this result will be for us, let us assume that $u \in \mathcal M^{even}_4$. By definition, the nodal set of $u$ is, in the upper right quadrant $Q^{\llcorner}$ of $\mathbb R^2$, asymptotic to an oriented half line
\[
\lambda : = r \, {\tt e}^\perp + \mathbb R \, {\tt e} ,
\] 
where  $r \in \mathbb R$ and ${\tt e} \in S^1$. We write ${\tt e } : =  (\cos \theta, \sin \theta)$ and, since we assume that the oriented line  $\lambda$ lies in $Q^{\llcorner}$, we have $\theta \in (0, \pi/2)$. 

\medskip

Given $R >0$, we define the plain triangle
\[
T_R := \left\{ (x,y) \in \mathbb R^2 \, : \,  x>0, \, y>0\quad \mbox{and} \quad (x,y) \cdot {\tt e} <R  \right\} ,
\]

\medskip

The divergence theorem implies that 
\begin{equation}
\int_{\partial T_R} \Xi (X,u)  \cdot\nu\, ds = 0 , \label{eq:lsr}
\end{equation}
where $\nu$ is the (outward pointing) unit normal vector field to $\partial T_R$.

\medskip

We set
\begin{equation}
c_0 := \int_{-\infty}^{+\infty} \left(  \frac{1}{2}Ê\, (H')^2  + F(H) \right) \,   ds.
\label{eq:c0}
\end{equation}
Taking $X := \partial_x$ and letting $R$ tend to infinity, we conclude, using the fact that $u$ is asymptotic to $\pm H$ along the half line $\lambda$, that 
\begin{equation}
c_0 \, \cos \theta  = \int_{x=0, \, y>0}    \left( \frac12 |\partial_y u|^2  + F(u) \right) \, dy  .
\label{eq:bc1}
\end{equation}
Observe that we have implicitly used the fact that $\partial_x u=0$ along the $y$-axis. Similarly, taking  $X := \partial_y$ and letting $R$ tend to infinity, we conclude that 
\begin{equation}
c_0 \, \sin \theta = \int_{y=0 , \,  x > 0}    \left( \frac12 |\partial_xu|^2  + F(u) \right) \, dx  .
\label{eq:bc2}
\end{equation}
Finally,  taking $X = x \, \partial_y - y \partial_x$ and letting $R$ tend to infinity, we get
\begin{equation}
c_0 \,r =   \displaystyle \int_{x=0, \, y>0}    \left( \frac12 |\partial_y u|^2  +F(u) \right)  y \, dy  -\int_{y=0, \, x>0} \left( \frac12 |\partial_x u|^2  +  F(u) \right) x \, dx .
\label{eq:bc3}
\end{equation}
The key observation is that it is possible to detect both the angle $\theta$ and the parameter $r$ which characterize the half line $\lambda$ just by performing some integration over the $x$-axis and the $y$-axis.  As one can guess this property will be very useful in the compactness analysis we are going to perform now. In some sense it will be enough to pass to the limit in the above integrals to guarantee the convergence of the parameters characterizing the asymptotics of the zero set of the solutions of (\ref{AC}).

\section{Properness}

In this section, we prove a compactness result for the set of $4$-ended solutions of (\ref{AC}). More precisely, we prove that, given $(u_n)_{n\geq 0}$, with $u_n \in \mathcal M^{even}_4$, a sequence of solutions of (\ref{AC}) whose angles $\theta_n := \pi/4 +  \mathcal P (u_n)$ converges to some limit angle $\theta_* \in (0, \pi/2)$, one can extract a subsequence which converges to a $4$-ended solution $u_*$ with angle $\theta_*$. Naturally, the fact that one can extract a subsequence which converges uniformly (at least on compacts of ${\mathbb R}^2$) to a solution $u_*$ of (\ref{AC}) is not surprising since the functions $u_n$ are uniformly bounded and, by elliptic regularity, have gradient which is uniformly bounded, hence this compactness result simply follows from the application of Ascoli-Arzela's Theorem. In general, it is hard to say anything about the limit solution $u_*$. It turns out that it is possible to control the zero set of the limit solution $u_*$ and prove that $u_*$ is also a $4$-ended solution. As we will see, a key ingredient in this analysis is provided by  the balancing formul{\ae}   defined in the previous section. 

\begin{theorem}
\label{teo compactness} Assume that we are given a sequence $(u_{n})_{n\geq 0}$, with $u_n \in \mathcal M^{even}_4$, which converge uniformly on compacts to a solution $u_*$. If $(\mathcal P(u_n))_{n\geq 0}$ converges in $\left( -\pi/4 ,\pi/4 \right)$, then, $u_{\ast} \in \mathcal M^{even}_4$, 
\[
\lim_{n\to \infty} \mathcal P(u_{n}) = \mathcal P(u_{\ast}) ,
\]
and 
\[
\lim_{n\to \infty} u_{n} =u_{\ast},
\] 
in the topology of $\mathcal{S}_{4}$.
\end{theorem}

The proof of this Theorem is decomposed into many small Lemma. Assume that we are given a sequence  $(u_n)_{n\geq 0}$, with $u_n \in \mathcal M^{even}_4$. Recall that
\begin{equation}
\partial_x u_n < 0 \qquad \mbox{and} \qquad \partial_y u_n > 0 ,
\label{mon}
\end{equation}
in the right upper quadrant
\[
Q^{\llcorner} : =\{(x,y) \in \mathbb R^2 \, : \, x>0 , \quad \mbox{and} \quad y >0 \} .
\]
We denote by 
\[
\mathcal Z_n : =  \{ (x,y) \in \mathbb R^2 \, :  \, u_n (x,y)=0 \} ,
\]
the nodal set of $u_n$. Monotonicity of $u_n$ in $Q^{\llcorner}$ implies that the zero set of $u_n$ is either a graph over the $x$-axis or a graph over the $y$-axis.  In particular, $\mathcal Z_n \cap \partial Q^{\llcorner}$ contains exactly one point which we denote by $p_n$
\[
\mathcal Z_n \cap \partial Q^{\llcorner} = \{ p_n\} .
\]

We define $\theta_n : = \pi/4 + \mathcal P(u_n)$ and $r_n$ to be the parameters describing the asymptotics of $\mathcal Z_n$ in the right upper quadrant $Q^{\llcorner}$. In other words, $\mathcal Z_nÊ\cap Q^{\llcorner}$ is asymptotic to the oriented half line
\[
\lambda_n : =  r_n \, {\tt e}_n^\perp + \mathbb R \, {\tt e}_n ,
\]
where ${\tt e}_n : =  (\cos \theta_n, \sin \theta_n)$. Finally, we assume that 
\[
\lim_{n \to \infty} \theta_n = \theta_* \in (0, \pi/2) .
\]

First, we prove that the point $p_n$ where the zero set of $u_n$ meets the boundary of the upper right quadrant $Q^{\llcorner}$ remains bounded as $n$ tends to infinity. 
\begin{lemma}
\label{lema compac 1} Under the above assumptions, the sequence $(p_n)_{n\geq 0}$ remains bounded, and hence converges. 
\end{lemma}
\begin{proof}
We argue by contradiction and for example assume that, up to a subsequence, $p_n = (0, y_n)$ with $\lim_{n\to \infty} y_n =+\infty$. 

\medskip

We define 
\[
w_{n} (x,y) : = u_{n} ( x ,y +y_n) ,
\]
Standard arguments involving elliptic estimates and Ascoli-Arzela's Theorem imply that, up to a subsequence, the sequence $(w_n)_{n\geq 0}$ converges uniformly on compacts to some function $w$ which is defined on $\mathbb R^2$ and which is again a solution of (\ref{AC}).  Since $\partial_y u_{n}>0$ for $y>0$, we conclude that 
\[
\partial_y w \geq 0
\]
in $\mathbb R^2$. Moreover, $w (  0,0)  =0$ and $\partial_x w(0,y) =0$ for all $y\in \mathbb R$  since  $u_{n}$ is even with respect to the $y$-axis.

\medskip

Observe that, since $w_n$ is not identically equal to $0$, we may apply Lemma \ref{le:4.1} to conclude that $w$ is not identically equal to $0$. Indeed, thanks to (\ref{mon}) we know that the zero set of $u_n$ is a graph over the $x$-axis for some function which is increasing and $u_n <0$ in the set $B_n := \{(x,y) \in \mathbb R^2 \, : \, |y|< y_n\}$. In particular, provided $n$ is chosen large enough, discs of arbitrary large radii can be inserted in $B_n$ and Lemma~\ref{mon} implies that  $|u_n^2-1| \leq 1/2$ in $\{(x,y) \in \mathbb R^2 \, : \, |y|< y_n - \rho_{1/2}\}$.

\medskip

Therefore, $w$ is bounded, monotone increasing in $y$, even with respect to the $y$-axis and according to De Giorgi's conjecture in dimension $2$ which is proven in \cite{MR1637919}, we conclude that 
\[
w (x,y)  = H (y) .
\]

Now, we use (\ref{eq:bc1}) which tells us that 
\[
c_0 \, \cos \theta_n  = \int_{x=0, \, y>0}    \left( \frac12 |\partial_y u_n|^2  + F(u_n) \right) \, dy  .
\]
Passing to the limit as $n$ tends to infinity, we conclude that 
\[
c_0 \, \cos \theta_*  = \lim_{n \to \infty} \int_{x=0, \, y>0}    \left( \frac12 |\partial_y u_n|^2  + F(u_n) \right) \, dy  .
\]
However, given $y_* >0$, 
\[
\begin{array}{rlll}
\displaystyle \int_{x=0, \, y>0}    \left( \frac12 |\partial_y u_n|^2  + F(u_n) \right) dy &  \geq \displaystyle \int_{x=0, \, y-y_n \in [-y_*, y_*]}  \left( \frac12 |\partial_y u_n|^2  + F(u_n) \right)  dy \\[3mm]
& = \displaystyle  \int_{x=0, \, y\in [-y_*,y_*]}    \left( \frac12 |\partial_y w_n|^2  + F(w_n) \right) dy ,
\end{array}
\]
for all $n$ large enough so that $y_n- y_* >0$.  Passing to the limit as $n$ tends to infinity, we conclude that 
\[
\lim_{n \to \infty} \int_{x=0, \, y>0}    \left( \frac12 |\partial_y u_n|^2  + F(u_n) \right) \, dy \geq  \int_{[-y_*,y_*]}    \left( \frac12 |H'(s)|^2  + F(H(s)) \right) \, ds .
\]
Hence
\[
c_0 \, \cos \theta_* \geq   \int_{[-y_*,y_*]}    \left( \frac12 |H'(s)|^2  + F(H(s)) \right) \, ds .
\]
Since $y_*$ can be chosen arbitrarily large, we get
\[
c_0 \, \cos \theta_* \geq   \int_{\mathbb R}   \left( \frac12 |H'(s)|^2  + F(H(s)) \right) \, ds = c_0 ,
\]
which is clearly in contradiction with the fact that $\theta_* >0$. If $p_n = (x_n,0)$ with $x_n$ tending to infinity, similar arguments using (\ref{eq:bc2}) and the fact that $\theta_*<\pi/2$. This completes the proof of the result. 
\end{proof}

\medskip

Thanks to the above Lemma, we know that $u_*$ is not identically constant equal to $0$ or $\pm 1$. Therefore, according to Theorem 4.4 in \cite{MR2381198} we know that the nodal set of $u_{\ast}$ in the upper right quadrant $Q^{\llcorner}$ must be a asymptotically a straight line which is not parallel to the $x$-axis nor to the $y$-axis. The Refined Asymptotic Theorem (Theorem 2.1 in \cite{dkp-2009}) then implies that $u_{\ast}$ is a $4$-ended solution. Some comment is due in the way  the Refined Asymptotic Theorem (Theorem 2.1 in \cite{dkp-2009}) is used. Indeed, in the statement of this result, one start with a solution of (\ref{AC}) which differs from the model heteroclinic solution sharing the same end by some $W^{2,2}$ function. Nevertheless, close inspection of the proof of Theorem 2.1 in \cite{dkp-2009} shows that the result remains valid provided we start from a solution which is asymptotic to the model heteroclinic solution in the $L^\infty$ sense, and this is precisely the situation in which we need the result.

\medskip

To proceed, we prove that  we have a uniform control on the nodal set of $u_n$ away from a compact set. Again, the balancing formul{\ae} will play an important role in the control of the nodal set of $u_n$. 

\medskip

To fix the ideas, we assume that the nodal set of $u_{n}$ in the upper right quadrant $Q^{\llcorner}$ is the graph of a function $y=f_{n} ( x)$. Since $u_n$ is a $4$-ended solution and given the notations introduced at the beginning of this section, we know that $f_n$ is asymptotic to the affine function given by
\[
\tilde f_n (x) : =  \tan \theta_n \, x +  \frac{r_n}{\cos \theta_n}.
\]
In particular, given $\delta >0$ and $n\geq 0$, there exists $x_{n,\delta}>0$ such that  
\[
| f_{n}^{\prime} (x)  - \tan \theta_n |  < \delta ,
\]
for all $x \geq x_{n,\delta}$. In the next Lemma, we prove that $x_{n,\delta}$ can be chosen to be independent of $n\geq 0$. In other words, this provides a uniform control on the derivative of $f_n$ away from a compact set.
\begin{lemma}
\label{twist} For all $\delta>0,$ there exists $x_\delta >0$ such that
\[
| f_{n}^{\prime} (x)  - \tan \theta_n |  < \delta ,
\]
for all $n\geq 0$ and for all $x \geq x_\delta$.
\end{lemma}
\begin{proof}
We now argue by contradiction. Observe that the result is true if we restrict our attention to a finite number of the $u_n$. Hence, if the result were not true, there would exist $\delta_* >0$ and sequences  $(x_k)_{k\geq 0}$ and $(n_k)_{k\geq 0} $ both tending to infinity such that 
\[
\sup_{x \geq x_k} | f_{n_k}^{\prime} (x)  - \tan \theta_{n_k} | \geq \delta_* ,
\]

We define $\bar x_k\geq x_k$ to be the supremum of the $x>0$ such that $|f_{n_k}^{\prime} (x)  - \tan \theta_{n_k} | \geq \delta_*$. Observe that $\bar x_k$ is well defined since 
\[
\lim_{x \to \infty} f_{n_k}^{\prime} (x) = \tan \theta_{n_k}.
\] 
By definition, we have 
\begin{equation}
| f_{n_k}^{\prime} (\bar x_k)  - \tan \theta_{n_k} | = \delta_* ,
\label{eq:pe}
\end{equation}
and 
\begin{equation}
\sup_{x\geq \bar x_k} | f_{n_k}^{\prime} (x)  - \tan \theta_{n_k} | \leq \delta_* .
\label{eq:se}
\end{equation}

Moreover, the sequence $(\bar x_k)_{k \geq 0}$ tends to infinity as $k$ tends to infinity and, if we define 
\[
\bar y_k : = f_{n_k} (\bar x_{k}), 
\] 
we find that the sequence $(\bar y_k)_{k\geq 0}$ also tends to infinity as $k$ does. This latter fact is just a consequence of the fact that $(u_n)_{n\geq 0}$ converges on compacts to $u_*$ which is a $4$-ended solution and hence the zero set of $u_*$ is the graph of a function which tends to infinity at infinity.

\medskip

We now consider the domain $D_k$ of $Q^{\llcorner}$ which contains the graph of $f_{n_k}$ for $x$ large enough and which is bounded by the half line $t \mapsto (0, \bar y_k- \bar x_k+t)$ for $t>0$, the segment joining $(0, \bar y_k - \bar x_k)$ to $(\bar x_k + \bar y_k,0)$ and the half line $t\mapsto (\bar x_k + \bar y_k + t, 0)$ for $t>0$. Observe that $\partial D_k$ contains the point $(\bar x_k, \bar y_k)$.  Applying the analysis of Section 4, we conclude that
\[
\int_{\partial Q^{\llcorner}} \Xi (\partial_x, u_{n_k})  \, \nu \, ds =  \int_{\partial D_k} \Xi (\partial_x, u_{n_k})  \cdot \nu \,  ds ,
\]
where $\nu$ denotes the outward pointing normal vector to the sets $Q^{\llcorner}$ and $D_k$. Using (\ref{eq:bc1}), we conclude that
\[
c_0 \, \cos \theta_{n_k}  =  \int_{\partial D_k} \Xi (\partial_x, u_{n_k})  \cdot \nu  \,  ds .
\]

Observe that, up to a subsequence, the sequence of functions 
\[
(x,y) \mapsto u_{n_k} (x+ \bar x_k, y + \bar y_k),
\]
converges, uniformly on compacts, to the heteroclinic solution whose zero set is the line passing through the origin, of slope $\lim_{k\to \infty} f'_{n_k} (\bar x_{k})$. Moreover, thanks to (\ref{eq:se}) we see that there exist constants $C>0$ and $\beta >0$, independent of $k\geq 0$, such that 
\[
|u^2_{n_k} ({\tt x}) -1|Ê+ |\nabla u_{n_k} ({\tt x})|\leq C \, e^{-\beta |{\tt x} - \bar {\tt x}_k|},
\]  
for all ${\tt x} \in \partial D_k$, where $\bar {\tt x}_{k} := (\bar x_k, \bar y_k)$. This property, together with the result of Lemma~\ref{le:2.2} allows one to conclude that 
\[
\lim_{k\to \infty}  \int_{\partial D_k} \Xi (\partial_x,u_{n_k})  \cdot \nu  \,  ds  =  c_0 \, \cos \tilde \theta_*,
\]
where $\tilde \theta_*$ is defined by 
\[
\tan \tilde \theta_* = \lim_{k\to \infty} f'_{n_k} (\bar x_k).
\]
This is clearly in contradiction with (\ref{eq:pe}) which implies that $|\tan \theta_*- \tan \tilde \theta_*|=\delta_*$. This completes the proof of the result.
\end{proof}

As a consequence, we have the~:
\begin{lemma}
\label{angle} 
Under the above assumptions, we have $\lim_{n\to \infty}\mathcal P (u_n) = \mathcal P (u_*)$. 
\end{lemma}

Since $u_n$ converges on compacts to $u_*$ which is a $4$-ended solution, we conclude, with the help of the previous Lemma that  the distance from a point ${\tt x} \in \mathcal Z_n$ to the $x$-axis and the $y$-axis, tends to infinity as $|{\tt x}|$ tends to infinity. We now prove a more quantitative version of this assertion in the~:
\begin{lemma}
\label{n3} There exists  constants $C >0$ and $\alpha>1$, such that
\[
\mathcal Z_n \cap Q^{\llcorner} \subset  \left\{ ( x,y )\in \mathbb R^2 \,  : \,  x>0, \, y>0, \quad \mbox{and} \quad \frac{x}{\alpha} - C\leq y\leq \alpha x +C \right\}  .
\]
\end{lemma}
\begin{proof}
According to Lemma~\ref{twist}, we have a uniform control on the slopes of the nodal sets of $u_n$ away from a tubular neighborhood of the $x$-axis and $y$-axis.  This means that these slopes are bounded away from $0$ and bounded independently of $n$. Next, in a ball of fixed radius, $u_{n}$ converges uniformly to $u_*$ and the result then follows at once.
\end{proof}

Recall that 
\[
\mathcal F (u_n) = (\theta_n-\pi/4, r_n).
\]
We set
\[
(\theta_*-\pi/4, r_*) := \mathcal F (u_*).
\]
Now that we have understood the behavior of the sequence $(\theta_n )_{n\geq 0}$, we turn to the behavior of the sequence $(r_n)_{n\geq 0}$, the other parameter which characterizes the asymptotic of the nodal set of $u_n$. We have the~:
\begin{lemma}
\label{shift} Under the above assumptions,  $\lim_{n \to \infty} r_{n}  = r_*$.
\end{lemma}
\begin{proof}
Again, the proof uses the balancing formula (\ref{eq:lsr}) but this time, we will use the vector field
\[
X=x\partial_y - y \partial_x .
\] 
Recall that (\ref{eq:bc3}) yields
\[
c_0 \, r_n =   \displaystyle \int_{x=0, \, y>0}    \left( \frac12 |\partial_y u_n|^2  +F(u_n) \right)  y \, dy  -\int_{y=0, \, x>0} \left( \frac12 |\partial_x u_n|^2  +  F(u_n) \right) x \, dx .
\]

The key ingredients are Lemma \ref{n3} and Lemma~\ref{le:2.2}, from which we get an exponential decay of the solution $u_n$ along the coordinate axis as $|{\tt x}|$ tends to infinity, the decay being uniform in $n\geq 0$. Once this decay is proven one uses the fact that $(u_n)_{n\geq 0}$ converges in $\mathcal C^1$ topology to $u_*$ uniformly on any given ball. 

\medskip

Using these remarks, one can pass to the limit as $n$ tends to infinity in the above equality to get to get
\[
c_0 \lim_{n \to \infty} r_n =   \displaystyle \int_{x=0, y>0}  \left( \frac12 |\partial_y u_*|^2  +F(u_*) \right)  y dy  -\int_{y=0, x>0} \left( \frac12 |\partial_x u_*|^2  +  F(u_*) \right) x  dx .
\]
Since the right hand side is equal to $c_0 r_*$, the proof is complete. \end{proof}

\medskip

At this point, we have shown that the sequence $(u_{n})_{n\geq 0}$ converges uniformly on compacts to $u_{\ast}$ and the ends of $u_n$ also converges to the end of $u_*$. However, this is not quite enough since our aim is to show the convergence of $(u_{n})_{n\geq 0}$ to $u_*$ in $\mathcal S_4$. 

\medskip

Recall that $\mathcal Z_n$, the zero set of $u_n$, is asymptotic to 
\[
\lambda_{n} := r_n \, {\tt e}_n^\perp + \mathbb R \, {\tt e}_n,
\]
in $Q^{\llcorner}$, where ${\tt e}_n = (\cos \theta_n, \sin \theta_n)$. We define $v_n$ in $Q^{\llcorner}$ by
\[
v_{n} \left( {\tt x} \right)  :=u_{n}\left(  \mathtt{x}\right)   - H \left(  {\tt x} \cdot {\tt e}^\perp_n -r_n \right)  .
\]
Then $\left\vert v_{n}\right\vert \to 0$ as $\left\vert \mathtt{x}\right\vert$ tends  to infinity in $Q^{\llcorner}$. 

\medskip

In the next Lemma, we prove that this convergence is in fact uniform in $n \geq 0$.
\begin{lemma}
As $|{\tt x}|$ tends to infinity, $\left\vert v_{n}(\mathtt{x})\right\vert $ converges to $0$ uniformly with respect to $n \geq 0$. 
\end{lemma}
\begin{proof}
The proof is by contradiction. If the result were not true, there would exists $\epsilon >0$, a sequence $(R_j)_{j\geq 0}$ tending to infinity, a sequence $({\tt x}_j)_{j\geq 0}$ such that $|{\tt x}_j|\geq R_j$ and a sequence $(n_j)_{j\geq 0}$ such that 
\begin{equation}
|v_{n_j} ({\tt x}_j)| \geq \epsilon.
\label{eq:zs}
\end{equation}
Up to a subsequence we can assume that $(\theta_{n_j}, r_{n_j})_{j\geq 0}$ converges to $(\theta_*, r_*)$. 

\medskip

Observe that the distance from ${\tt x}_j$ to $\lambda_{n_j}$ is necessarily bounded since, according to Lemma~\ref{le:2.2}, $v_{n_j}$ tends to $0$ away from $\lambda_{n_j}$. Let $\bar {\tt x}_j$ be the orthogonal projection of  ${\tt x}_j$ onto $\lambda_{n_j}$. 

\medskip

Making use of elliptic estimates and Aslcoli-Arzela's Theorem, we can assume, up to a subsequence, that $(u_{n_j} (\cdot - \bar {\tt x}_j))_{j\geq 0}$ converges uniformly on compacts to a solution of (\ref{AC}) which is non trivial and which, thanks to the de Giorgi conjecture, is an heteroclinic solution $\bar u$ of (\ref{AC}). The end of this heteroclinic solution is the affine line of angle $\bar \theta$.  As in the proof of Lemma~\ref{n3}, we use the vector field $X= \partial_x$ in the  balancing formula to conclude that $\theta_* = \bar \theta$. 

\medskip

Therefore, the parameters of the end of $\bar u$ are given by $(\theta_*, \bar r)$.  As in the proof of Lemma~\ref{n3}, we  use the vector field $X = x\partial_y - y \partial_x$ in the balancing formula to conclude that $r_* = \bar r$. This is clearly a contradiction with (\ref{eq:zs}). 
\end{proof}

\medskip

Thanks to the Refined Asymptotics Theorem (Theorem 2.1 in \cite{dkp-2009}) we can decompose
\[
u_n = v_n + u_{\lambda_n},
\]
and 
\[
u_* = v_* + u_{\lambda_*},
\]
where $\lambda_n, \lambda_* \in \Lambda^4_{ord}$ and where $v_n, v_* \in e^{-\delta (1+|{\tt x}|^2)} \, W^{2,2} (\mathbb R^2)$ for some $\delta >0$. Observe that, {\it a priori}, the parameter $\delta$ can vary with $n$ but close inspection of the proof of Theorem 2.1 in \cite{dkp-2009})  shows that $\delta >0$ can indeed be chosen independently of $n \geq 0$ since the ends $\lambda_n$ converge to a fixed end $\lambda_*$.

\medskip

This, together with the fact that $(u_n)_{n\geq 0}$ converges uniformly on compacts to $u_*$ implies that $(u_n)_{n\geq 0}$ converges to $u_*$ in the topology of $\mathcal{S}_{4}$.  This completes the proof of the properness of the classifying map $\mathcal P$. 

\medskip

Let $M$ be the connected component of $\mathcal M^{even}_4$ which contains the saddle solution. We claim that the properness of $\mathcal P$ implies that the image by $\mathcal P$ of $M$ is the entire interval $(-\pi/4,\pi/4)$. The proof of this claim goes as follows~: we argue by contradiction and assume that $\mathcal P :  M \to (-\pi/4, \pi/4)$ is not onto. Recall that if $u \in \mathcal M^{even}_4$, then $\bar u$ defined by
\[
\bar u (x,y) : = -u (y,x),
\]
also belongs to $\mathcal M_4^{even}$ and $M$ is also invariant under this transformation. We will write $\bar u = J \, u$. The properness of $\mathcal P$ implies that $M$ is compact and one dimensional. Hence, it must be diffeomorphic to $S^{1}$.  Obviously $J : M \to M$ is a diffeomorphism and the saddle solution is a fixed point of $J$. Since $M$ is diffeomorphic to $S^{1}$, there must be at least another fixed element $v \in M$ which is a fixed point of $J$. Then, the zero set of $v$ is union of the two lines $y = \pm x$.  But, according to \cite{MR1198672} or \cite{MR2381198}, a solution of (\ref{AC}) having as zero set the two lines $y=\pm x$ is the saddle solution. This is a contradiction and the proof of the claim is complete. Note that this argument does not guarantee that there are no other compact connected components in $\mathcal M_4^{even}$. To prove this fact, we will need one more result which will be described in the next section. In any case, instead of using the argument outlined above to show that $\mathcal P$ is onto, one can use the next section of the paper.

\medskip
\section{Any connected component of $\mathcal M_4^{even}$ is not compact}
\label{rigidity}

We have shown in section \ref{sec no degeneracy} that elements in $\mathcal M^{even}_4$ are even-nondegenerate. According to the moduli space theory for solutions of (\ref{AC}) (see Section 8 and Theorem 2.2 in \cite{dkp-2009}), any connected component of $\mathcal M_4^{even}$ is a one dimensional manifold and its image by $\mathcal F$ is a smooth (possibly immersed) curve in $(-\pi/4, \pi/4) \times \mathbb R$.  In particular, any compact connected component $M \subset {\mathcal{M}}_{4}^{even}$ would have to be diffeomorphic to $S^{1}$. In this section, we show that this cannot happen. 
\begin{theorem}
All connected components of  $\mathcal M_4^{even}$ are not compact, namely, there is no closed loop in $\mathcal M_4^{even}$.
\label{th:6.1}
\end{theorem}
\begin{proof}
We argue by contradiction and assume that $\mathcal M^{even}_4$ contains a connected component $M$ which is diffeomorphic to $S^1$.  We choose a smooth regular parameterization of $M$ by
\[
\sigma \in S^1 \mapsto u(\cdot,\sigma) \in M,
\] 
so that
\[
\Delta u(\cdot, \sigma)-F^{\prime}\left(  u(\cdot, \sigma)\right)  = 0,
\]
for all $\sigma \in S^1$ and $\partial_\sigma u \neq 0$ for all $\sigma \in S^1$. Differentiation with respect to $\sigma$ implies that $\partial_{\sigma}u \in T_u \mathcal S_4$ satisfies
\[
\left( \Delta -F^{\prime\prime}\left(  u\right)\right)  \, \partial_{\sigma}u = 0.
\]

Observe that, for all ${\tt x} \in \mathbb R^2$,
\[
0 = u\left(  {\tt x}, 2\pi \right) - u \left(  {\tt x} , 0\right)  = \int_{0}^{2\pi} \partial_{\sigma} u \left( {\tt x} ,\sigma\right)  \, d\sigma.
\]
Choosing ${\tt x}$ to be the origin, this implies that there exists $\sigma_{*} \in S^1$, such that $\partial_{\sigma} u((0,0),\sigma_{*}) =0$. We define 
\[
\phi : = \partial_{\sigma} u(\cdot ,\sigma_{*}) .
\]
Observe that $\phi\neq 0$ and that $\phi$ is even with respect to the symmetry about both the $x$-axis and the $y$-axis.

\medskip

By definition, any element $u$ of $\mathcal S_4$ can be decomposed into the sum of a function in $W^{2,2} (\mathbb R^2)$ and an element of the form $u_\lambda$ as defined in (\ref{def w}). Moreover, because of the symmetries, $u_\lambda$ only depends on the two parameters $r$ and $\theta$ which characterize $\lambda$.   In particular,  the tangent space of $\mathcal S_4$ at $u$ can be decomposed as
\[
T_u\mathcal S_4 =  W^{2,2}\left( \mathbb{R}^{2}\right) \oplus{\mathfrak{D}},
\] 
where
\[
\mathfrak D  := \mbox{Span}\, \{\partial_r u_\lambda , \partial_\theta u_\lambda \}.
\]
It is easy to check that $\partial_\theta u_\lambda$ is linearly growing along the zero set of $u$ while $\partial_r u_\lambda $ is bounded.

\medskip

Since $\phi (0,0) =0$ and since $\phi$ is symmetric with respect to the $x$-axis and the $y$-axis, there exists $\Omega$, a nodal domain of $\phi$,  which is included in one of the four half spaces $\{(x,y)\in \mathbb R^2 \, : \, \pm x >0\}$ or $\{(x,y)\in \mathbb R^2 \, : \, \pm y >0\}$. We claim that this nodal domain can be chosen so that $\phi$ is bounded on it.  Indeed, if $\phi \in W^{2,2}\left( \mathbb{R}^{2}\right) \oplus \mbox{Span}\, \{\partial_r u_\lambda \}$, then $\phi$ is bounded and one can select any nodal domain contained in a half space. 

\medskip

The other case to consider is the case where $\phi = a \, \partial_\theta u_\lambda + \tilde \phi$ where $\tilde \phi$ is bounded. Inspection of $\partial_\theta u_\lambda $ near the end of $u$ shows that, away from a large ball $B(0,R)$, the function $\phi$ does not vanish along the zero set of $u$. In this case, it is enough to select a nodal domain of $\phi$ which is unbounded and which, away from $B(0,R)$,  does not contain the zero set of $u$. It is easy to check that $\phi$ is bounded in such a nodal domain. 

\medskip

For example, let us assume that the nodal domain $\Omega \subset \{(x,y) \in \mathbb R^2 \, : \, x >0\}$. Then, one can repeat the argument of Step 2 in the proof of Theorem~\ref{th:ng}, with $\psi = \partial_x u$, to prove that 
\begin{equation}
\int_\Omega |\nabla h|^2 \psi^2 \zeta_R^2 d{\tt x} \leq 2 \left( \int_{\Omega \cap A_R} |\nabla h|^2 \psi^2  \zeta^2_R d{\tt x}\right)^{1/2} \left( \int_{\Omega \cap A_R} \phi^2 |\nabla \zeta_R|^2 d{\tt x}\right)^{1/2} ,
\label{eq:ss}
\end{equation}
where $h : = \frac{ \phi}{\psi}$. Using the fact that $\phi$ is bounded, and letting $R$ tend to infinity, we conclude that 
\[
\int_\Omega |\nabla h|^2 \psi^2 \, d{\tt x} <+\infty.
\]
Using this information back into (\ref{eq:ss}), and letting $R$ tend to infinity, we conclude that 
\[
\int_\Omega |\nabla h|^2 \psi^2  d{\tt x} =0,
\]
and this implies that $\phi \equiv 0$ in $\Omega$. The unique continuation theorem then implies that $\phi \equiv 0$, which is a contradiction. 
\end{proof}

\medskip

We observe that from the above considerations, we can give a different proof of Theorem~\ref{th:4}. Indeed, we choose $M$ to be the connected component of ${\mathcal{M}}_{4}^{even}$ which contains the saddle solution. Using the Implicit Function Theorem (Theorem 2.2 in \cite{dkp-2009}), which applies since we have proven that any element of $\mathcal M_4^{even}$ is non even-degenerate, we conclude that $M$ is a smooth, one dimensional manifold. By Theorem~\ref{th:6.1},  $M$ is necessarily non compact and the image of $M$ by $\mathcal P$ cannot be compact either. Hence the image of $M$ by $\mathcal P$ contains either an interval of the form $(-\pi/4, \delta)$ or $(\delta, {\pi}/{4})$. Since the image by $\mathcal P$ of the saddle solution is $0$, we conclude that $(-\pi/4, \delta)$ or $(\delta, {\pi}/{4})$ contains $0$. Moreover, the image of $M$ by $\mathcal P$ is symmetric with respect to $0$ and hence it has to be the whole interval $(-\pi/4, \pi/4)$.

\section{The Morse index of $4$-ended solutions}

In this section, we give a proof of Theorem~\ref{th:2.8}. The proof follows from the result of Proposition~\ref{pr:3.1} together with a result of \cite{Bap-Dev}. For the sake of completeness, we give here a straightforward proof which is inspired from \cite{Don} and which is independent from the proof of \cite{Bap-Dev}. 

\medskip

Assume that $u$ is a $2k$-ended solution. We consider the linearized operator about $u$
\[
L : =  - \Delta + F''(u) .
\]
From Proposition~\ref{pr:3.1}, we know that there exists $R_0 >0$ and a positive function $\Psi >0$ defined in $\mathbb R^2$ such that 
\[
L \Psi \leq 0 , 
\]
in $\mathbb R^2- B_{R_0}$. 

\medskip

Let $\phi$ be an eigenfunction of $L$ in $B_R$ (with $0$ Dirichlet boundary conditions), which is associated to a negative eigenvalue, namely
\[
L \phi = \lambda \, \phi,
\]
in $B_R$ with $\psi=0$ on $\partial B_R$ and $\lambda <0$. For all $R >R_0$, one can use the function $\Psi$ as a barrier to show, using the maximum principle, that there exists a constant $C >0$ independent of $R\geq 2 R_0$ such that
\begin{equation}
\| \phi \|_{L^\infty (B_R - B_{R_0})} \leq C \,  \| \phi \|_{L^\infty (\partial B_{R_0})}Ê.
\label{eq:1}
\end{equation}

Now, elliptic estimates also imply that there exists a constant $C >0$, which does not depend on $R > 2R_0$ such that
\begin{equation}
\| \phi \|_{L^\infty (B_{R_0})} \leq C \,  \| \phi \|_{L^2 (B_{2R_0})}Ê.
\label{eq:2}
\end{equation}
Observe that, in order to obtain this inequality, we have implicitly used the fact that the negative eigenvalues of $L$ are bounded. Using both (\ref{eq:1}) and (\ref{eq:2}), we conclude that there exists a constant $C >0$ which does not depend on $R> 2R_0$ such that, for all eigenfunctions of $L$ in $B_R$ which are associated to negative eigenvalues, we have
\begin{equation}
\| \phi \|_{L^\infty (B_{2R_0})} \leq C \,  \| \phi \|_{L^2 (B_{2R_0})}Ê.
\label{eq:est}
\end{equation}

The proof now follows the strategy of \cite{Don} to estimate the multiplicity of the eigenvalues of the Laplace-Beltrami operator starting from H\"ormander's estimate. Let $\phi_1, \ldots, \phi_m$ be an orthonormal basis (in $L^2(B_{2R_0})$) of the vector space spanned by restrictions of the eigenfunctions of $L$ associated to negative eigenvalues. We define the Bergman kernel associated to the orthogonal projection in $L^2(B_{2R_0})$ onto $V$. Namely,  
\[
K({\tt x}, {\tt y}) : = \sum_{j=1}^m \phi_j({\tt x}) \phi_j({\tt y}).
\]
Observe that $K$ is independent of the choice of the orthonormal basis. Also 
\[
m = \int_{B_{2R_0}}ÊK({\tt x} ,{\tt x})  \, d{\tt x},
\]
is the dimension of $V$. Obviously, there exists ${\tt x}_0 \in \bar B_{2R_0}$ such that
\[
K({\tt x}_0, {\tt x}_0) \mbox{Vol} ( B_{2R_0} )  \geq m .
\]
We then consider the evaluation form
\[
\mathcal E_{{\tt x}_0}  (\phi) : = \phi({\tt x}_0) ,
\]
and choose the orthonormal basis $\phi_1, \ldots, \phi_m $ such that $\mathcal E_{{\tt x}_0} (\phi_j) =0$ for $j=2, \ldots, m$. Then 
\[
K({\tt x}_0, x_0) \,  \mbox{Vol} ( B_{2R_0}) = \phi_1({\tt x}_0)^2 \,  \mbox{Vol} ( B_{2R_0} ) \geq m 
\]
But, (\ref{eq:est}) implies that 
\[
\|\phi_1 \|_{L^\infty (B_{2R_0})} \leq  C  \,    \| \phi_1 \|_{L^2(B_{2R_0})}   = C .
\]
Therefore  $m \leq C^2 \, \mbox{Vol} ( B_{2R_0} )$ and hence the dimension of $V$ is bounded independently of $R> 2 R_0$. 

\medskip

This implies that the Morse index of $L$ is finite.

\end{document}